\documentclass[11pt,a4paper]{article}
\usepackage{indentfirst}
\usepackage{amsfonts}
\usepackage{amssymb}
\usepackage{mathrsfs}
\usepackage{amsmath}
\usepackage{amsthm}
\usepackage{enumerate}
\usepackage{cite}
\usepackage{mathrsfs}
\allowdisplaybreaks
\usepackage{geometry}
\usepackage{ifpdf}
\ifpdf
\usepackage[colorlinks=true,linkcolor=blue,citecolor=red, final,backref=page,hyperindex]{hyperref}
\else
\usepackage[colorlinks,final,backref=page,hyperindex,hypertex]{hyperref}
\fi

\usepackage{txfonts}
\usepackage{indentfirst}
\usepackage{latexsym}
\usepackage[all]{xy}

\def\<{\langle}
\def\>{\rangle}
\def\a{\alpha}

\newtheorem{thm}{Theorem}[section]
\newtheorem{lem}[thm]{Lemma}

\newtheorem{prop}[thm]{Proposition}
\newtheorem{ex}[thm]{Example}

\theoremstyle{definition}
\newtheorem{defn}{Definition}[section]

\theoremstyle{remark}
\newtheorem{re}{Remark}[section]
\begin{document}
	\title{\bf On the Hom-Lie CoDer pairs}
	\author{\bf Basdouri Imed, Sadraoui Mohamed Amin, Assif Sania}
	\author{{ Asif Sania $^{1}$
			\footnote { Corresponding author,  E-mail:  11835037@zju.edu.cn}
			Basdouri Imed $^{2}$
			\footnote { Corresponding author,  E-mail:  basdourimed@yahoo.fr}
			,\  Sadraoui Mohamed Amin $^{3}$
			\footnote { Corresponding author,  E-mail: aminsadrawi@gmail.com}
			,\  Saghrouni Abdelkader $^{3}$
			\footnote { Corresponding author,  E-mail: Abdelkadersaghrouni@gmail.com}
		}\\
		{\small 1.  Nanjing University of Information Science and Technology, 210044, Nanjing, PR China.}\\
		{\small 2.  University of Gafsa, Faculty of Sciences Gafsa, 2112 Gafsa, Tunisia}\\
		{\small 3. University of Sfax, Faculty of Sciences Sfax,  BP
			1171, 3038 Sfax, Tunisia}\\
		{\small 4.  University of Sfax, Faculty of Sciences Sfax,  BP
			1171, 3038 Sfax, Tunisia}
	}
	\date{}
	\maketitle
	\begin{abstract}The present research paper investigates the intricate fields of Hom-Lie algebra  and Hom-Lie coalgebra, providing a complete analysis of their key concepts and important examples. Precisely, the paper introduces the concept of Hom-Lie coderivation pairs and demystifies its duality with Hom-Lie derivation pairs, inspecting pertinent facts such as representation and semi-direct product. Furthermore, the study examines the connection between Hom-Lie, pre-Lie, and Ass-Coder pairs with the use of crucial operators such as commutator, Rota-Baxter operator, and endomorphism operator. Finally, the paper concludes by presenting the construction of Hom-pre-Lie coderivation pairs through a dual to an endomorphism operator.
	\end{abstract}
	
	\textbf{Key words}:\, $\alpha$-Coderivation, Rota-Baxter operators, endomorphism operators
	
	\textbf{Mathematics Subject Classification}
	
	\numberwithin{equation}{section}
	\tableofcontents
	\section{Introduction}The notion of the algebraic Hom-structures was first introduced by Jonas T. Hartwig, Daniel Larson, and Sergei Silvestrov in \cite{H1} where they develop an approach to deformations of the Witt and Virasoro algebra based on $\sigma$-derivation. In the same paper, they define a Hom-Lie algebra $(L,\xi)$ as a non associative algebra $L$ together with an algebra homomorphism $\xi:L \rightarrow L$ such that the Jacobi identity is switched as follow 
\begin{equation*}
    \circlearrowleft_{x,y,z}[(id+\xi)(x),[y,z]_L]_L=0.
\end{equation*}
Theory of algebraic Hom-structures was developed in \cite{A2, M6, Y1, Y2}, where the Jacobi identity was introduced with a linear commuting structural map $\a$, as follows
\begin{equation*}
    \circlearrowleft_{x,y,z}[\alpha(x),[y,z]_L]_L=0.
\end{equation*}
Since 2007 the algebraic Hom structures had attracted many researchers in different fields and extended in various algebras like Lie algebras, associative algebras, and Leibniz algebras. In 1980 Michaelis introduce Lie coalgebras as the dual of Lie algebras leading to the study of Lie bialgebras, quantum groups, and other fields like Harrison's cohomology of commutative algebras, homotopy theory, and noncommutative geometry (see \cite{M2} for more details). Based on the theory of Lie coalgebras and the Hom structures, Makhlouf Abdenacer and Sergei Silvestrov introduced Hom-coalgebras in \cite{M5}. \\
The notion of algebraic Hom structures was also studied to connect different algebras to each other by the commutator bracket or by some operators like Rota-Baxter and endomorphism operators. \\
The following diagrams show some connections between Hom-Lie, Hom pre-Lie, and Hom-coassociative coalgebras.
\begin{center}
	$\xymatrix{\textbf{\emph{pre-Hom-Lie coalgebras pair}}&\underrightarrow{ commutator} &\textbf{\emph{Hom-Lie coalgebras}} }$
\end{center}
\begin{center}
	$\xymatrix{
		\textbf{\emph{Hom-coassociative coalgebras pair}}&\underrightarrow{ commutator} &\textbf{\emph{Hom-Lie coalgebras}} }$
\end{center}
\begin{center}
	$\xymatrix{ \textbf{\emph{pre-Hom-Lie coalgebras pair}}&\underrightarrow{ R.B.O~(\lambda=0)} &\textbf{\emph{Hom-Lie coalgebras}} }$
\end{center}
The notion of Hom-algebras is not studied until now in  reference to the derivation pairs. In \cite{T2} authors introduced the Lie Der pairs structure which is defined as a Lie algebra equipped with a derivation and studied its cohomology and its related deformation. Inspired by the work of Michaelis in \cite{M2}, Asif et. al in \cite{A1}, dualize the notion of Lie-Der pairs to the Lie-CoDer pair (Lie coalgebra equipped with a coderivation). Where the connections between Lie CoDer pairs, pre-Lie CoDer pairs, and AssCoDer pairs (Ass CoDer pairs were studied in \cite{D2}) were studied in detail. In this paper, we extend the structure of Lie CoDer pairs to the  \textbf{Hom-Lie CoDer} pairs and provided useful constructions. In particular,  In Section $2$, we recall some important definitions and examples related to Hom-Lie algebras and Hom -Lie coalgebras to support our results in the later sections. In Section $3$, we define the Hom-Lie coderivation pair, which is dual to a Hom-Lie derivation pair. We describe this concept with the help of some examples and provide important propositions and lemmas. Moreover, we study some basic results like representation and \textbf{semidirect product} and we show that the dual of \textbf{Hom-Lie Der pairs} is just the \textbf{Hom-Lie CoDer pairs} and vice versa. In Section 4, we study the connection between \textbf{Hom Lie, pre-Lie, Ass Coder pairs} via the commutator $\Delta=\Delta_0-\tau \circ \Delta_0$, Rota-Baxter, and endomorphism operators. In particular, we combine the Rota-Baxter operator with the Hom-LieDer pair and dualize it to the case of the Hom-Lie CoDer pair which leads us to define the Rota-Baxter Hom-LieDer pair, the Rota-Baxter Hom-Lie CoDer pair, and Rota-Baxter Hom-pre-Lie CoDer pair. We also constructed Hom-pre-Lie CoDer pair with the help of a dual to an endomorphism operator. In this paper, we consider all the vector spaces and tensor products are over the field $\mathbb{K}$ of characteristic $0$. 
\section{Preliminaries} For any linear spaces $V$ and $W$ we have the flip map $\tau_{(V, W)}: V \otimes W \rightarrow W \otimes V$ defined by $v \otimes w \mapsto w \otimes v$ for $v \in V, w\in W$ if there is no confusion we simply denote $\tau$ instead of $\tau_{(V, W)}$. Denote by $\xi$ the cyclical permutation map given by 
 \begin{align*}
 	\xi&:x \otimes y \otimes z \mapsto y \otimes z \otimes x \\
 	\xi^2&:x \otimes y \otimes z \mapsto z \otimes x \otimes y \\
 	\tau^{12}&:x \otimes y \otimes z \mapsto y \otimes x \otimes z
 \end{align*}
 \begin{defn} \label{def1.1} \cite{M2}
 	A Lie coalgebra over $\mathbb{K}$  is a pair $(L,\Delta)$ of a linear space $L$ and a linear map $\Delta : L \rightarrow L \otimes L$ where the following two conditions holds 
 	\begin{enumerate}
 		\item [1)] $Im(\Delta) \subset Im(1-\tau)$,
 		\item [2)] $(1+\xi+\xi^2) \circ (1 \otimes \Delta) \circ \Delta=0:L \rightarrow L \otimes L \otimes L$.
 	\end{enumerate}
 \end{defn}
Here the map $\Delta$ is called the Lie co-bracket of $L$, we can write the two previous conditions respectively 
\begin{enumerate}
	\item [3)] $\Delta=- \tau \circ \Delta$,
	\item [4)] $(1 \otimes \Delta) \circ \Delta - (\Delta \otimes 1) \circ \Delta +(1 \otimes \tau) \circ (\Delta \otimes 1) \circ \Delta=0.$ 
\end{enumerate}

 \begin{defn} \label{def1.2} \cite{S1}
 	An associative coalgebra over $\mathbb{K}$ is a pair $(L,\Delta)$ where $L$ is a vector space and $\Delta :L \rightarrow L \otimes L$ is a linear map such that 
 	$$ (1 \otimes \Delta)\Delta=(\Delta \otimes 1) \Delta$$
 \end{defn}
 The notion of Hom-coalgebras was introduced first in \cite{M5}. 
 A Hom-coalgebra is a triple $(L,\Delta,\alpha)$ where $L$ is a vector space, $\Delta:L \rightarrow L\otimes L$ and $\alpha:L \rightarrow L$ are linear maps.
 \begin{defn} \label{def1.3} \cite{M5}
 \begin{enumerate}
     \item [1)]  A Hom-coassociative coalgebra is a Hom-coalgebra $(L,\Delta,\alpha)$ such that the following identity holds 
     \begin{equation*}
         (\alpha \otimes \Delta)\circ \Delta=(\Delta \otimes \alpha)\circ \Delta 
     \end{equation*}
      \item [2)] A Hom-Lie coalgebra is a Hom-coalgebras $(L,\Delta,\alpha)$ such that the following identity holds 
    \begin{equation*}
        (1+\xi+\xi^2) \circ (\alpha \otimes \Delta) \circ \Delta=0
    \end{equation*}
    \end{enumerate}
 \end{defn}
 Remark that if $\alpha:L \rightarrow L$ satisfy  $(\alpha \otimes \alpha) \circ \Delta=\Delta \circ \alpha$, then we say that $(L,\Delta,\alpha)$ is a multiplicative Hom-coalgebra.\\
Let $(L,\Delta,\alpha)$ be a Hom-coalgebra with an $\alpha$-coderivation $\varphi_L$ then a new coalgebra structure $(L,\Delta,\alpha,\varphi_L)$ rises which it is denoted by \textbf{Hom-CoDer pair} . The word pair meant Hom-coalgebra $L$ and $\alpha$-coderivation $\varphi_L$. where $\alpha$-coderivation $\varphi_L$ is a linear map of $L$ such that $\Delta \circ \varphi_L=(\varphi_L \otimes \alpha) \Delta+(\alpha \otimes \varphi_L) \Delta$ and the set of all $\alpha$-coderivation of $L$ is denoted by $CoDer_{\alpha}(L)$ . If $\Delta$ denotes a Lie co-bracket then we  obtain a \textbf{Hom-Lie CoDer pair} and if $\Delta$ is an associative co-bracket then we obtain a \textbf{Hom-AssCoDer pair}. \\
Let's recall some results about Hom-coalgebras. 
\begin{defn}{(Hom-coalgebras morphism)} \label{def1.4} \cite{B1} \\
	Let $(L,\Delta,\alpha)$ and $(L',\Delta',\alpha')$ be two hom-coalgebras (Lie or associative). A linear map $h :L \rightarrow L'$ is a morphism of Hom-coalgebras (Lie or associative) if and only if 
	\begin{equation*}
		(h \otimes h)\Delta=\Delta'\circ h \text{ and } h \circ \alpha=\alpha \circ h.
	\end{equation*} 
\end{defn}
\begin{defn} \cite{B1}  \label{def1.5}
	Let $(L,\Delta,\alpha)$ be an Hom-Lie coalgebra and $(M,\beta)$ be a Hom-module, if there exists a linear map $\rho_{\beta}:M \rightarrow L \otimes M$ such 
	\begin{enumerate}
		\item[i)] $\rho_{\beta} \circ \beta=(\alpha \otimes \beta) \circ \rho_{\beta},$ 
		\item[ii)] $(\Delta \otimes \beta) \circ \rho_{\beta}=(\alpha \otimes \rho_{\beta}) \circ \rho_{\beta}-(\tau \otimes id_M)(\alpha \otimes \rho_{\beta}) \circ \rho_{\beta}.$
	\end{enumerate}
Then $M$ is called a $L$-comodule.
\end{defn}
\begin{defn} \cite{M3} \label{1.6}
	Let $\lambda \in \mathbb{K}$, a triple $(L,\Delta,R)$ is called a Rota-Baxter coalgebra of weight $\lambda$ if $(L,\Delta)$ is a coalgebra and $R$ is a linear map from $L$ to $L$ such that for all $l \in L$ 
	\begin{equation} \label{eq0.1}
		R(l_1) \otimes R(l_2)=R(l)_1\otimes R(R(l)_2)+R(R(l)_1) \otimes R(l)_2+\lambda R(l)_1 \otimes R(l)_2.
	\end{equation}
Here $R$ is called a Rota-Baxter operator.
\end{defn}
\begin{re}
	A Rota-Baxter operator $R$ on a Hom-coalgebra satisfies the same equation \eqref{eq0.1} including $R \circ \alpha=\alpha \circ R$.
\end{re}
Throughout this paper we use Sweedler's notation \cite{S1}
\begin{equation*}
	\Delta(l)=l_0 \otimes l_1,\text{ for } l,l_0,l_1 \in L.
\end{equation*}
For example: $(\alpha \otimes \Delta)\Delta(l)=(\alpha \otimes \Delta)(l_1 \otimes l_2)=\alpha(l_1) \otimes l_{21} \otimes l_{22}$ .
	\section{ Hom-Lie CoDer pairs} \label{sec3}
Due to the significant importance of Hom-algebra structures, all the results obtained in the classical algebras (Lie algebras, associative algebras, Leibniz algebras, etc) are twisted, via a linear map $\alpha$ of $L$, to the case of Hom-algebraic structures. Furthermore, As a dual case to Hom algebra, many researchers have investigated the Hom-coalgebra structure in \cite{B1, C1, Y1} .\\
In this section, we introduce the structure of \textbf{Hom-Lie CoDer pairs} and study some basic results on it.

	\begin{defn} \label{def2.1}
		A  \textbf{Hom-LieCoDer pair} consists of a  Hom-Lie coalgebra $(L,\Delta,\alpha)$ and an $\alpha$-coderivation $\varphi_L$. 
		We denoted by $(L,\Delta,\varphi_L,\alpha)$. 
	\end{defn}
 We say that $(L,\Delta,\varphi_L,\alpha)$ is a multiplicative \textbf{Hom-LieCoDer pair} if $(L,\Delta,\alpha)$ is a multiplicative Hom-Lie coalgebra.
\begin{defn} \label{def2.2}
	Let $(L_1,\Delta_1,\varphi_{L_1},\alpha)$ and $(L_2,\Delta_2,\varphi_{L_2},\alpha)$ are two \textbf{Hom-Lie CoDer pairs}, a  morphism $\psi:L_1 \rightarrow L_2$ is an Hom-Lie coalgebra morphism (see definition \eqref{def1.4}) if it satisfies the following equations 
	\begin{equation}
		\psi \circ \varphi_{L_1}=\varphi_{L_2} \circ \psi.
		\label{eq2.1}
	\end{equation}
\end{defn}
In \cite{M1} Shahn Majid defines the Lie comodule (or right Lie comodule) of Lie coalgebra $(L,\Delta)$, it is further studied in \cite{D1} in reference to cohomology theory. Based on the stated  references, we define the \textbf{Hom-Lie CoDer} comodule as follow
\begin{defn} \label{def2.3}
	Let $(L,\Delta,\varphi_L,\alpha)$ be an \textbf{Hom-Lie CoDer pair}, a \textbf{Hom-LieCoDer} comodule $(M,\rho_{\beta},\varphi_M,\beta)$ is a qudruple, in whiich $(M,\rho_{\beta},\beta)$ as an Hom-Lie comodule of $(L,\Delta,\alpha)$ and a linear map $\varphi_M:M \rightarrow M$ is a coderivation map, satisfying
	\begin{equation}
		\label{eq2.2}
		\rho_{\beta} \circ \varphi_M(m)=(\varphi_M(m) \otimes \alpha ) \rho_{\beta}(m)+(\beta \otimes \varphi_L) \rho_{\beta}(m), \text{ for }m \in M.
	\end{equation} 
\end{defn}
\begin{ex} \label{ex2.1}
	Let $(L,\Delta,\varphi_L,\alpha)$ be an \textbf{Hom-Lie CoDer pair}. Then $(L \otimes L,\rho_{ad_{\alpha}},\varphi_L,\alpha)$ is an Hom-Lie-CoDer comodule of $L$, with 
	\begin{equation}
		\label{eq2.3}
		\rho_{ad_{\alpha}}=(\alpha \otimes \Delta)-\xi \circ (\Delta \otimes \alpha): L \otimes L \rightarrow L \otimes L \otimes L \otimes L.
	\end{equation}
\end{ex}
In the next proposition, we investigate the case of  direct sum $(L \oplus M)$, where $(L,\Delta,\varphi_L,\alpha)$ is the \textbf{Hom-Lie CoDer pair} and $(M,\rho_{\beta},\varphi_M,\beta)$  is its \textbf{Hom-Lie CoDer} comodule.
\begin{prop}
	\label{prop2.2}
Let $(L,\Delta,\varphi_L,\alpha)$ be an Hom-Lie-CoDer pair, $M$ is a linear space, $\rho_{\beta} :M \rightarrow M \otimes L$ and $\varphi_M : M \rightarrow M$ a linear map. Then $(M,\rho_{\beta},\varphi_M,\beta)$ is an Hom-Lie-CoDer comodule of $L$ if and only if $(L \oplus M,\widetilde{\Delta},\widetilde{\varphi},\widetilde{\alpha})$ is an Hom-Lie-CoDer pair such that
\begin{align*}
	\widetilde{\Delta}(x,m)&=\Delta(x)+\rho_{\beta}(m)- \tau \circ (\rho_{\beta}(m)), \\
	\widetilde{\varphi}(x,m)&=(\varphi_L(x),\varphi_M(m)), \\
	\widetilde{\alpha}(x,m)&=(\alpha(x),\beta(m)),
\end{align*}
where $x \in L $ and $m \in M$. 
Which is called \textbf{ semidirect product } of the \textbf{Hom-Lie CoDer pair} $(L,\Delta,\varphi_L,\alpha)$ and its comodule.
\end{prop}
\begin{proof}
For	the first part of the proof see \cite{D1}, we need just to prove that $\widetilde{\varphi}$ is an $\widetilde{\alpha}$-coderivation. \\ Let $(x,m) \in (L \oplus M)$.
\begin{align*}
	\widetilde{\Delta} \circ \widetilde{\varphi}(x,m)&=\widetilde{\Delta}(\varphi_L(x),\varphi_M(m)) \\
	&=\Delta(\varphi_L(x))+\rho_{\beta}(\varphi_M(m))- \tau \circ (\rho_{\beta}(\varphi_M(m))) .
\end{align*}
On the first hand, we have 
\begin{equation} \label{eq2.4}
	(\widetilde{\varphi} \otimes \widetilde{\alpha}) \circ \widetilde{\Delta}(x,m) = (\varphi_L \otimes \alpha) \Delta(x)+(\varphi_M \otimes \alpha) (\rho_{\beta}(m))-(\varphi_M \otimes \alpha) (\tau \circ (\rho_{\beta}(m))).
\end{equation}
On the second hand, we have
\begin{equation}\label{eq2.5}
	(\widetilde{\alpha} \otimes  \widetilde{\varphi} ) \circ \widetilde{\Delta}(x,m) = (\alpha \otimes  \varphi_L) \Delta(x)+( \beta \otimes \varphi_L ) (\rho_{\beta}(m))-( \alpha_M \otimes \varphi_L ) (\tau \circ (\rho_{\beta}(m))).
\end{equation}
Now we combine  equations \eqref{eq2.4} and \eqref{eq2.5} and using \eqref{eq2.2}, we obtain
\small{
\begin{align*}
	(\widetilde{\varphi} \otimes \widetilde{\alpha}) \circ \widetilde{\Delta}+(\widetilde{\alpha} \otimes  \widetilde{\varphi} ) \circ \widetilde{\Delta}(x,m)&= 
	(\varphi_L \otimes \alpha + \alpha \otimes  \varphi_L) \Delta(x) + (\varphi_M \otimes \alpha + \beta \otimes \varphi_L) (\rho_{\beta}(m))- (\varphi_M \otimes \alpha + \beta \otimes \varphi_L )(\tau \circ (\rho_{\beta}(m))) \\
	&=\Delta \circ \varphi_L(x)+ \rho_{\beta} \circ \varphi_M(m)-\tau \circ (\rho_{\beta} \circ \varphi_M(m)) \\
	&=\widetilde{\Delta} \circ \widetilde{\varphi}(x,m).
\end{align*}
}
This completes the proof.
\end{proof}	
\begin{prop} \label{prop2.3} 
	Let $(L_1,\Delta_1,\varphi_{L_1},\alpha_1)$ and $(L_2,\Delta_2,\varphi_{L_2},\alpha_2)$ be two \textbf{Hom-Lie-CoDer pairs}, so the quadruple $(L_1 \oplus L_2,\Delta_{L_1 \oplus L_2},\varphi_{L_1 \oplus L_2},\alpha_{L_1 \oplus L_2})$ is a \textbf{Hom-Lie-CoDer pair} with :
	\begin{eqnarray}
		\Delta_{L_1 \oplus L_2}:L_1 \oplus L_2 \rightarrow (L_1 \oplus L_2) \otimes (L_1 \oplus L_2) \ &;& \ (x,y) \mapsto \Delta_1(x)+ \Delta_2(y),\label{eq2.6} \\
		\varphi_{L_1 \oplus L_2} : L_1 \oplus L_2 \rightarrow L_1 \oplus L_2 \ &;& \ (x,y) \mapsto (\varphi_{L_1}(x),\varphi_{L_2}(y)),\label{eq2.7} \\
		\alpha_{L_1 \oplus L_2}: L_1 \oplus L_2 \rightarrow L_1 \oplus L_2 \ &;& \ (x,y) \mapsto (\alpha_1(x),\alpha_2(y)).\label{eq2.8}
	\end{eqnarray}
 Where $x \in L_1$ and $y \in L_2$
\end{prop}	

Lie coalgebra was preferably introduced by Walter Michaelis in \cite{M2} as dually to Lie algebra and the inverse way is true too (Lie coalgebra gives rise to Lie algebra by dualizing). This means to define \textbf{Hom-Lie-CoDer pair} properly we must study its dual relation with \textbf{Hom-Lie-Der pair}. \\
So we introduce the notion of \textbf{Hom-LieDer pairs} in which
	we consider only the multiplicative Hom-Lie algebras.
\begin{defn} \label{def2.4}
	A multiplicative \textbf{Hom-LieDer pair} consists of a multiplicative Hom-Lie algebra $(L,[\cdot,\cdot],\alpha)$ and an $\alpha$-derivation $\varphi_L \in Der_{\alpha}(L)$, we denote \textbf{Hom-LieDer pair}
	$(L,[\cdot,\cdot],\varphi_L,\alpha)$ or simply $(L, \varphi_L,\alpha)$.
\end{defn}

In the sequel, we consider only multiplicative Hom-LieDer pair.

\begin{defn} \label{def2.5}
	Let $(L,[\cdot,\cdot]_L,\varphi_L,\alpha )$ and $(K,[\cdot,\cdot]_K,\varphi_{K},\beta )$ be two \textbf{Hom-LieDer pairs}, a morphism from $L$ to $K$ is a Hom-Lie algebra morphism $H : L \rightarrow K$ such that
	\begin{equation} \label{eq2.9}
		H \circ \varphi_L = \varphi_{K} \circ H.
	\end{equation}
\end{defn}
Let's define the representation of \textbf{Hom-LieDer pair}
\begin{defn}
	A representation of $(L,\varphi_L,\alpha)$ on a vector space $V$ is a pair $(\rho_A,\varphi_V)$ where $\rho_A : L \rightarrow gl(V)$, with respect to $A \in gl(V)$ and $\varphi_V : V \rightarrow V$ a linear map, such that $\rho_A$ is a representation of Hom-Lie algebra $L$ on $V$ and satisfies :
	\begin{equation} \label{eq2.10}
		\varphi_V \circ \rho_A(x) - \rho_A(\alpha(x)) \circ \varphi_V=\rho_A(\varphi_L(x)) \circ A
	\end{equation}
	we denote such a representation by $(V;\rho_A,\varphi_V)$
\end{defn}
\begin{ex}
	$(L;ad,\varphi_L)$ is a representation of the \textbf{Hom-LieDer pair} $(L,\alpha,\varphi_L)$ which is called the adjoint representation of $(L,\varphi_L,\alpha)$.
\end{ex}
If we define the representation of \textbf{Hom-LieDer pair} it is necessary to define the semi-direct product of the \textbf{Hom-LieDer pair} and its representation
\begin{prop} Let $(L,\varphi_L,\alpha)$ be a \textbf{Hom-LieDer pair} and $(V;\rho_{\beta},\varphi_V)$ be a representation of it. Then $(L \oplus V, \varphi_L + \varphi_V,\alpha + \beta)$ is a \textbf{Hom-LieDer pair} where the bracket on the Hom-Lie algebra  $(L \oplus V)$ is defined by : \\
	$[x+X,y+Y]_{L \oplus V}=\Big([x,y], \rho_{\beta}(x)(Y)-\rho_{\beta}(y)(X) \Big)$.\\
	and $(\alpha + \beta)$ is defined by : \\
	\begin{center}
		$\alpha + \beta : L \oplus V \rightarrow L \oplus V$;$(x,X) \mapsto \alpha(x) + \beta(X)$
	\end{center}
and $\varphi_L +\varphi_V$ is defined by :
\begin{center}
	$\varphi_L +\varphi_V: L\oplus V \rightarrow L \oplus V \ ; \ (x,X) \mapsto \varphi_L(x)+\varphi_V(X).$
\end{center} for all $x,y\in L,~~ X,Y \in V$. Such a \textbf{Hom-LieDer pair} is called the \textbf{semi-direct product} of $(L,\varphi_L,\alpha)$ and $(V;\rho_{\beta},\varphi_V)$.
\end{prop}
\begin{proof}
	We need just to prove that $(\varphi_L + \varphi_V)$ is an $(\alpha + \beta)$-derivation on $(L \oplus V)$.

	\begin{align*}
		\varphi_L + \varphi_V([x+X,y+Y]_{L \oplus V}) 
		&=\varphi_L + \varphi_V \big([x,y]+\rho_{\beta}(x)(Y)-\rho_{\beta}(y)(X) \big) \\
		&=\varphi_L([x,y])+\varphi_V(\rho_{\beta}(x)(Y))-\varphi_V(\rho_{\beta}(y)(X)) \\
		&=[\varphi_L(x),\alpha(y)]+\rho_{\beta}(\varphi_L(x)) \circ \beta(Y) -\rho_{\beta}(\alpha(y)) \circ \varphi_V(X) \\
		&+[\alpha(x),\varphi_L(y)]+\rho_{\beta}(\alpha(x)) \circ \varphi_V(Y) -\rho_{\beta}(\varphi_L(y)) \circ \beta(X)  \\
		&=[(\varphi_L+\varphi_V)(x+X),(\alpha + \beta)(y+Y)]_{L \oplus V}
		+[(\alpha + \beta)(x+X),(\varphi_L+\varphi_V)(y+Y)]_{L \oplus V}.
	\end{align*}
	This completes the proof.
\end{proof}
	
In the next, proposition we show that the direct sum of two different \textbf{Hom-LieDer pairs} is also an \textbf{Hom-LieDer pair}, (see \cite{S2}  for more details).
\begin{prop} \label{prop2.6}
	Given two \textbf{Hom-LieDer pairs} $(L,[\cdot,\cdot]_L,\varphi_L,\alpha)$ and $(K,[\cdot,\cdot]_{K},\varphi_{K},\beta)$, there is a \textbf{Hom-LieDer pair} $(L \oplus K,[\cdot,\cdot]_{L \oplus K},\varphi_L +\varphi_{K},\alpha + \beta)$ where the skew-symmetric bilinear map \\ $[\cdot,\cdot]_{L \oplus K} \wedge ^2 (L \oplus K) \rightarrow L \oplus K$ is given by
	\begin{center}
		$[(x_1,y_1),(x_2,y_2)]_{L \oplus K}=([x_1,x_2]_L,[y_1,y_2]_{K}),$~~ for all $x_1,x_2 \in L$; $y_1,y_2 \in K$\end{center}

	and a linear map $\alpha + \beta : L \oplus K \rightarrow L \oplus K$ defined by 
	 $(x,y) \mapsto \alpha(x) + \beta(y)$, where $x \in L, y \in K$. 

Also $\varphi_L +\varphi_K$ is defined by :
\begin{center}
	$\varphi_L +\varphi_K: L\oplus K \rightarrow L \oplus K \ ; \ (x,y) \mapsto \varphi_L(x)+\varphi_K(y).$
\end{center}
\end{prop}
\begin{proof}
	By \cite{S2}, we have $(L \oplus K,[\cdot,\cdot]_{L \oplus K},\alpha + \beta)$ is a Hom-Lie algebra. We need just to prove that $(\varphi_L+\varphi_{K})$ is an $(\alpha + \beta)$-derivation. \\
	\begin{align*}
		&\varphi_L +\varphi_{K} ([(x_1,y_1),(x_2,y_2)]_{L \oplus K}) \\
		&=\varphi_L +\varphi_{K} ([x_1,x_2]_L,[y_1,y_2]_{K}) \\
		&=\varphi_L ([x_1,x_2]_L) + \varphi_{K} ([y_1,y_2]_{K}) \\
		&=[\varphi_L (x_1),\alpha (x_2)]_L+ [\alpha (x_1),\varphi_L (x_2)]_L + [\varphi_{K}(y_1), \beta (y_2)]_{K}+[\beta (y_1),\varphi_{K}(y_2)]_{K} \\
		&=([\varphi_L (x_1),\alpha (x_2)]_L,[\varphi_{K}(y_1), \beta (y_2)]_{K})+([\alpha (x_1),\varphi_L (x_2)]_L,[\beta (y_1),\varphi_{K}(y_2)]_{K}) \\
		&=[\varphi_L +\varphi_{K}(x_1,y_1),(\alpha + \beta)(x_2,y_2)]_{L \oplus K} + [(\alpha + \beta)(x_1,y_1),\varphi_L +\varphi_{K}(x_2,y_2)]_{L \oplus K}.
	\end{align*}
	This completes the proof.
\end{proof}
Let $(L,\mu,\alpha)$ be an Hom-Lie algebra, then $(L^*,\mu^*,\alpha)$ is an Hom-Lie coalgebras by the following relation
\begin{center}
	$<\mu^*(\eta),x \otimes y>=<\eta,\mu(x \otimes y)>$, for $\eta \in L^*$ and $x,y \in L$.
\end{center}  (see \cite{Y1} for more details).
Moreover, let $(L,\Delta,\alpha)$ be an Hom-Lie coalgebra, then $(L^*,\Delta^*,\alpha)$ is an Hom-Lie algebra with :
\begin{center}
	$<\Delta^*(\eta_1 \otimes \eta_2),x>=<\eta_1 \otimes \eta_2,\Delta(x)>$, for $\eta_1,\eta_2 \in L^*$ and $x \in L$
\end{center}
The previous two relations lead us to the next result 
\begin{prop} \label{prop2.7}
	Let $(L,\Delta,\varphi_L,\alpha)$ be a \textbf{Hom-Lie CoDer pair}, so $(L^*,\Delta^*,\varphi_L^*,\alpha)$ is a \textbf{Hom-Lie Der pair}.
\end{prop}
\begin{proof}
	We need just to prove that $\varphi_L^*$ is an $\alpha$-derivation.
	Assume that for $\eta_1,\eta_2 \in L^*$ and $x \in L$
	\begin{align*}
		<\varphi_L^* \circ \Delta^*(\eta_1 \otimes \eta_2),x>&=<\Delta^*(\eta_1 \otimes \eta_2),\varphi_L(x)> \\
		&=<\eta_1 \otimes \eta_2,\Delta \circ \varphi_L(x)> \\
		&=<\eta_1 \otimes \eta_2,(\alpha \otimes \varphi_L) \circ \Delta(x)+(\varphi_L \otimes  \alpha ) \circ \Delta(x)> \text{( because } \varphi_L \text{ is } \alpha \text{-coderivation)} \\
		&=<\eta_1 \otimes \eta_2,(\alpha \otimes \varphi_L) \circ \Delta(x)> + <\eta_1 \otimes \eta_2,(\varphi_L \otimes  \alpha ) \circ \Delta(x)> \\
		&=<(\alpha \otimes \varphi_L^*)(\eta_1 \otimes \eta_2),\Delta(x)>+<(\varphi_L^* \otimes \alpha )(\eta_1 \otimes \eta_2),\Delta(x)> \\
		&=<\Delta^* \circ (\alpha \otimes \varphi_L^*+\varphi_L^* \otimes \alpha)(\eta_1 \otimes \eta_2),x>.
	\end{align*}
Which mean $\varphi_L^*$ is an $\alpha$-derivation. This completes the proof.
\end{proof}
\begin{ex}
From propositions \eqref{prop2.7} and \eqref{prop2.6}, $(L_1 \oplus L_2,\Delta_{L_1 \oplus L_2},\varphi_{L_1 \oplus L_2},\alpha_{L_1 \oplus L_2})$ is a \textbf{Hom-Lie-CoDer pair}, then $(L_1^* \oplus L_2^*,\Delta_{L_1^* \oplus L_2^*},\varphi^*_{L_1^* \oplus L_2^*},\alpha_{L_1 \oplus L_2})$ is a \textbf{Hom-Lie Der pair}.	
\end{ex}
It is well known that there exists a natural injection $L^* \otimes L^* \rightarrow (L \otimes L)^*$, so we can write 
\begin{equation} \label{eq2.11}
	[f,g](l)=(f \otimes g)\Delta(l)=f(l_1)g(l_2);\text{ where } l,l_1,l_2 \in L.
\end{equation}
\begin{prop}
	Let $(L,\Delta,\varphi_L,\alpha)$ be an \textbf{Hom-Lie CoDer pair}, so using \eqref{eq2.11} we obtain  a \textbf{Hom-Lie-Der pair} with $(L^*,\Delta^*,\varphi_L^*,\alpha)$ 
	\begin{align*}
		\Delta^*(f,g)&=[f,g], \\
		\varphi_L^*:L^* \rightarrow L^* \ ; \ &\eta \mapsto \varphi_L^*(\eta)=\eta \circ \varphi_L,\\
		\alpha \circ \eta&=\eta \circ \alpha.
	\end{align*}
For $f,g,\eta \in L^*$.
\end{prop}
\begin{proof}
	Following \cite{M2}, we just need to prove that $\varphi_L^*$ is an $\alpha$-coderivation.\\
	Let $l \in L$ and $f,g \in L^*$ such that $\Delta(l)=l_1 \otimes l_2$
	\begin{align*}
		\varphi_L^*[f,g](l)&=[f,g] \circ \varphi_L(l)\\
		&=(f \otimes g) \circ \Delta \circ \varphi_L(l) \\
		&=(f \otimes g) \big((\varphi_L \otimes \alpha) \circ \Delta(l)+(\alpha \otimes \varphi_L) \circ \Delta(l) \big) \\
		&=f(\varphi_L(l_1))g(\alpha(l_2))+f(\alpha(l_1))g(\varphi_L(l_2)) \\
		&=\varphi_L^*(f)(l_1) \alpha(g(l_2))+\alpha(f(l_1))\varphi_L^*(g)(l_2) \\
		&=[\varphi_L^*(f),\alpha(g)](l)+[\alpha(f),\varphi_L^*(g)](l).
	\end{align*} 
This completes the proof.
\end{proof}The following lemma is very essential to study the cohomology of \textbf{Hom-LieCoDer pair}. However, in this paper, we are not studying cohomology in detail but it will be useful for studying cohomology of \textbf{Hom-LieCoDer pair} in other papers. 
\begin{lem}
	Let $(M,\rho_{\beta},\varphi_M,\beta)$ be an \textbf{Hom-Lie CoDer} comodule of the \textbf{Hom-Lie CoDer pair} $(L,\Delta,\varphi_L,\alpha)$. Then $(M^*,\rho_{\beta}^*,\varphi_M^*,\beta)$ is an \textbf{Hom-LieDer} module of  $(L^*,\Delta^*,\varphi_L^*,\alpha)$ in such a way that the following equations hold
 \begin{eqnarray}\begin{aligned}\rho_{\beta}^*(f,\gamma):=&\rho_{\beta}^*(f)(\gamma)=f\cdot \gamma=-(\gamma \otimes f) \circ \rho_{\beta}\label{2.12}\\ \varphi_L^*(f)=& f \circ \varphi_L, \label{eq2.13} \\
	\varphi_M^*(\gamma)=& \gamma \circ \varphi_M. \label{eq2.14}
	\end{aligned}
\end{eqnarray} 
For $f \in L^*, \gamma \in M^*.$
\end{lem}
\begin{proof}
	Let $f,g \in L^*$ , $\gamma \in M^*$ and $m \in M$, we have $(M,\rho_{\beta},\beta)$ is an Hom-Lie comodule of $(L,\Delta,\alpha)$ which mean $\rho_{\beta} \circ \beta = (\alpha \otimes \beta) \circ \rho_{\beta}$. \\
	Let us begin by proving that $(M^*,\rho_{\beta}^*,\beta)$ is a $L^*$-module and then we prove that $(M^*,\rho_{\beta}^*,\varphi_M^*,\beta)$ is $L^*$-Hom-Lie-Der module. In the first hand, we prove that
	\begin{align*}
		<\beta \circ \rho_{\beta}^*(f,\gamma),m>&=<(f,\gamma),\rho_{\beta} \circ \beta(m)> \\
		&=<(f,\gamma),(\alpha \otimes \beta) \circ \rho_{\beta}(m)> \\
		&=<\rho_{\beta}^* \circ (\alpha \otimes \beta)(f,\gamma),m > \\
		&=<\rho_{\beta}^*(\alpha(f),\beta(\gamma)),m>
	\end{align*}
For the second hand 
\begin{align*}
	<\rho_{\beta}^*(\Delta^*(f,g),\beta(\gamma)),m>&=<\rho_{\beta}^*(\Delta^* \otimes \beta)((f,g),\gamma),m> \\
	&=<((f,g),\gamma),(\Delta \otimes \beta) \circ \rho_{\beta} (m)> \\
	&=<\rho_{\beta}^*(\alpha \otimes \rho_{\beta}^*)((f,g),\gamma),m>-<\rho_{\beta}^*(\alpha \otimes \rho_{\beta}^*)(\tau^* \otimes id_{M^*})((f,g),\gamma),m> \\
	&=<\rho_{\beta}^*(\alpha(f),\rho_{\beta}^*(g,\gamma)),m>-<\rho_{\beta}^*(\alpha(g),\rho_{\beta}^*(f,\gamma)),m>
\end{align*}
with the previous two results we obtain that $(M^*,\rho_{\beta}^*,\beta)$ is a $L^*$-module. \\
The second part of the proof we have 
\begin{align*}
	\varphi_M^* \circ \rho_{\beta}^*(f)(\gamma)&=\rho_{\beta}^*(f)(\gamma) \circ \varphi_M \\
	&=-(\gamma \otimes f) \circ (\rho_{\beta} \circ \varphi_M) \\
	&\text{ by using \eqref{eq2.2} we obtain } \\
	&=-(\gamma \otimes f) \circ (\varphi_M \otimes \alpha) \circ \rho_{\beta}-(\gamma \otimes f) \circ (\beta \otimes \varphi_L) \circ \rho_{\beta} \\
	&=-((\gamma \circ \varphi_M) \otimes (f \circ \alpha)) \circ \rho_{\beta} - ((\gamma \circ \beta) \otimes f \circ \varphi_L) \circ \rho_{\beta} \\
	&=-(\varphi_M^*(\gamma) \otimes \alpha(f)) \circ \rho_{\beta} - (\beta(\gamma) \otimes \varphi_L^*(f)) \circ \rho_{\beta} \\
	&= \rho_{\beta}^*(\varphi_L^*(f)) \beta(\gamma) + \rho_{\beta}^*(\alpha(f)) \varphi_M^*(\gamma).
\end{align*}
This completes the proof.
\end{proof}
\section{Rota-Baxter operators on Hom-Lie-CoDer pairs} \label{sec4}
The Rota-Baxter operator was first  introduced by G. Baxter to solve many analytic and combinatorial problems which applied to many fields in mathematics and mathematical physics. At that time, most of the study on these operators was on the associative algebras and then it was extended to the case of Lie algebras and  pre-Lie algebras. Later on, these operators are studied with reference to dual algebra. In this section, we extend this relation to the case of \textbf{Hom-Lie CoDer pairs} and \textbf{pre-Hom-Lie CoDer pair}.
\begin{defn} \label{def3.1}
	A Hom-pre-Lie coalgebra is a triple $(L,\Delta_0,\alpha)$ such the following identity holds
	\begin{equation}
		(1_{L^{\otimes^3}}-\tau^{12})((\Delta_0 \otimes \alpha) \circ \Delta_0-(\alpha \otimes \Delta_0) \circ \Delta_0)=0
	\end{equation}
Here $1_{L^{\otimes^3}}$ is the identity map on $L^{\otimes^3}$
\end{defn}
\begin{defn} \label{def3.2}
	A \textbf{Hom-pre-Lie CoDer pair} is a quadruple $(L,\Delta_0,\varphi_L,\alpha)$ where $(L,\Delta_0,\alpha)$ is a Hom-pre-Lie coalgebra and $\varphi_L$ is an $\alpha$-coderivation such that 
	\begin{center}
		$\Delta_0 \circ \varphi_L=(\varphi_L \otimes \alpha + \alpha \otimes \varphi_L) \circ \Delta_0$.
	\end{center}
\end{defn}
In the next theorem, we show that the Hom-pre-Lie CoDer pair can be connected to the Hom-Lie CoDer pair by using the commutator relation.
\begin{thm} \label{thm3.1} We can generate a Hom-Lie CoDer pair $(L,\Delta,\varphi_L,\alpha)$ from 
	a \textbf{Hom-pre-Lie CoDer pair} $(L,\Delta_0,\varphi_L,\alpha)$,by using the  the co-bracket $\Delta=\Delta_0-\tau \circ \Delta_0$ .           
\end{thm}
\begin{proof} Similar to the proof of Lemma 2.2.1 in \cite{T1}, it is easy to prove that $\varphi_L$ is an $\alpha$-coderivation.
\end{proof}
We can describe the previous lemma in the following diagram
\begin{center}
	$\xymatrix{
		\textbf{\emph{Hom-pre-Lie CoDer pair}}&\underrightarrow{commutator} &\textbf{\emph{Hom-Lie CoDer pair}} }$
\end{center}
Next, we prove that every Rota-Baxter \textbf{Hom-CoAss-CoDer pair} of weight $(\lambda=-1)$ gives a \textbf{Hom-pre-Lie CoDer pair}. For this, we first introduce the following definitions.
\begin{defn}
	A $5$-tuple $(L,\Delta,\varphi_L,\alpha,R)$ is said to be a Rota-Baxter \textbf{Hom-CoDer pair} of weight $\lambda$ if $(L,\Delta,\varphi_L,\alpha)$ is a \textbf{Hom-CoDer pair} and $R$ is a Rota-Baxter operator of weight  $\lambda$, verifying the equation \eqref{eq0.1}, and the following condition
	\begin{equation} \label{eq3.2}
		\varphi_L \circ R=R \circ \varphi_L
	\end{equation}
\end{defn} 
\begin{thm} \label{thm3.2}
	Let $(L,\Delta,\varphi_L,\alpha,R)$ be a Rota-Baxter \textbf{Hom-CoAss CoDer pair} of weight $(\lambda=-1)$. We define $\tilde{\Delta}$ on $L$ by 
	\begin{equation*}
	\tilde{\Delta}=(R \otimes 1_L) \Delta - \tau \circ (1_L \otimes R)\Delta-\Delta,	
	\end{equation*}
that implies:
	\begin{equation*}
	\tilde{\Delta}(l)=R(l_1) \otimes l_2 - R(l_2) \otimes l_1 - l_1 \otimes l_2.	
	\end{equation*} for  all $l \in L$.
Then $(L,\tilde{\Delta},\varphi_L,\alpha)$ is a \textbf{Hom-pre-Lie CoDer pair}.
With \begin{equation} \label{eq3.3}
	R \circ \alpha=\alpha \circ R,
\end{equation}
and the equation \eqref{eq3.2} hold.	
\end{thm}
To prove the above-stated theorem, we  use the results in the following proposition.
\begin{prop} \label{prop3.3}
	Let be $(L,\Delta,\varphi_L,\alpha)$ a \textbf{Hom-CoAss CoDer pair} and $R : L \rightarrow L$ a Rota-Baxter operator. Then along with equations \eqref{eq3.2} and \eqref{eq3.3} the following two identities holds 
	\begin{align*}
		(R \otimes 1_L)(\alpha \otimes \varphi_L) \Delta(l)&=(\alpha \otimes \varphi_L)(R \otimes 1_L)\Delta(l), \\
		\tau \circ (1_L \otimes R)(\alpha \otimes \varphi_L) \Delta(l)&=(\varphi_L \otimes \alpha)(\tau \circ (1_L \otimes R))\Delta(l).
	\end{align*}
\end{prop}
\begin{proof}
	Consider that  $l \in L$,
	\begin{align*}
		(R \otimes 1_L)(\alpha \otimes \varphi_L) \Delta(l)
        &=(R \otimes 1_L)(\alpha \otimes \varphi_L)(l_1 \otimes l_2) \\
		&=(R \otimes 1_L)(\alpha(l_1) \otimes \varphi_L(l_2)) \\
		&=R(\alpha(l_1)) \otimes \varphi_L(l_2) \\
		&\overset{\eqref{eq3.3}}{=} \alpha(R(l_1)) \otimes \varphi_L(l_2) \\
		&=(\alpha \otimes \varphi_L)(R \otimes 1_L)\Delta(l).
	\end{align*}
\begin{align*}
	\tau \circ (1_L \otimes R)(\alpha \otimes \varphi_L) \Delta(l)&=\tau \circ (1_L \otimes R)(\alpha \otimes \varphi_L)(l_1 \otimes l_2) \\
	&=\tau \circ (1_L \otimes R)(\alpha(l_1) \otimes \varphi_L(l_2)) \\
	&=R(\varphi_L(l_2)) \otimes \alpha(l_1) \\
	&\overset{\eqref{eq3.2}}{=}\varphi_L(R(l_2)) \otimes \alpha(l_1) \\
	&=(\varphi_L \otimes \alpha)(R(l_2) \otimes 1_L(l_1)) \\
	&=(\varphi_L \otimes \alpha)(\tau \circ (1_L \otimes R))(l_1 \otimes l_2) \\
	&=(\varphi_L \otimes \alpha)(\tau \circ (1_L \otimes R))\Delta(l).
\end{align*}
\end{proof}
The proof of theorem \eqref{thm3.2} is given as follows:
\begin{proof}
	We need to prove that $\tilde{\Delta}_L - \tau^{12} \tilde{\Delta}_L=0$ \\
	where $\tilde{\Delta}_L=(\tilde{\Delta} \otimes \alpha)\tilde{\Delta}-(\alpha \otimes \tilde{\Delta}) \tilde{\Delta}$ and $\tau^{12}$ as defined above. We begin by computing $\tilde{\Delta}_L(l)$ with $l \in L$
	\begin{align*}
		\tilde{\Delta}(l)&=R(R(l_1)_1) \otimes R(l_1)_2 \otimes \alpha(l_2)-R(R(l_1)_2) \otimes R(l_1)_1 \otimes \alpha(l_2)-R(l_1)_1\otimes R(l_1)_2 \otimes \alpha(l_2) \\
		&-R(R(l_2)_1) \otimes R(l_2)_2 \otimes \alpha(l_1)+R(R(l_2)_2) \otimes R(l_2)_1 \otimes \alpha(l_1)+R(l_2)_1\otimes R(l_2)_2 \otimes \alpha(l_1) \\
		&-R(l_{11}) \otimes l_{12} \otimes \alpha(l_2)+R(l_{12}) \otimes l_{11} \otimes \alpha(l_2)+l_{11} \otimes l_{12} \otimes \alpha(l_2)-\alpha(R(l_1)) \otimes R(l_{21}) \otimes l_{22} \\
		&+\alpha(R(l_1)) \otimes R(l_{22}) \otimes l_{21}+\alpha(R(l_1)) \otimes l_{21} \otimes l_{22}+\alpha(R(l_2)) \otimes R(l_{11}) \otimes l_{12}-\alpha(R(l_2)) \otimes R(l_{12}) \otimes l_{11} \\	
		&-\alpha(R(l_2)) \otimes l_{11} \otimes l_{12} +\alpha(l_1) \otimes R(l_{21}) \otimes l_{22}- \alpha(l_1) \otimes R(l_{22}) \otimes l_{21}-\alpha(l_1) \otimes l_{21} \otimes l_{22} \\
		&=R(R(l_1)_1) \otimes R(l_1)_2 \otimes \alpha(l_2)-R(R(l_1)_2) \otimes R(l_1)_1 \otimes \alpha(l_2)-R(l_1)_1\otimes R(l_1)_2 \otimes \alpha(l_2) \\
		&-R(R(l_2)_1) \otimes R(l_2)_2 \otimes \alpha(l_1)+R(R(l_2)_2) \otimes R(l_2)_1 \otimes \alpha(l_1)+R(l_2)_1\otimes R(l_2)_2 \otimes \alpha(l_1) \\
		&+R(\alpha(l_2)) \otimes \alpha(l_1) \otimes \alpha(l_3)-\alpha(R(l_1)) \otimes R(\alpha(l_2)) \otimes \alpha(l_3) \\
		&+\alpha(R(l_1)) \otimes R(\alpha(l_3)) \otimes \alpha(l_2)+\alpha(R(l_3)) \otimes R(\alpha(l_1)) \otimes \alpha(l_2)-\alpha(R(l_3)) \otimes R(\alpha(l_2)) \otimes \alpha(l_1) \\
		&-\alpha(R(l_3)) \otimes \alpha(l_1) \otimes \alpha(l_2) +\alpha(l_1) \otimes R(\alpha(l_2)) \otimes \alpha(l_3)- \alpha(l_1) \otimes R(\alpha(l_3)) \otimes \alpha(l_2)\\
	\end{align*}
Now 
\begin{align*}
	(\tilde{\Delta}_L - \tau^{12} \tilde{\Delta}_L(l)&=R(R(l_1)_1) \otimes R(l_1)_2 \otimes \alpha(l_2)-R(R(l_1)_2) \otimes R(l_1)_1 \otimes \alpha(l_2)-R(l_1)_1\otimes R(l_1)_2 \otimes \alpha(l_2) \\
	&-R(R(l_2)_1) \otimes R(l_2)_2 \otimes \alpha(l_1)+R(R(l_2)_2) \otimes R(l_2)_1 \otimes \alpha(l_1)+R(l_2)_1\otimes R(l_2)_2 \otimes \alpha(l_1) \\
	&-\alpha(R(l_1)) \otimes R(\alpha(l_2)) \otimes \alpha(l_3)-\alpha R(l_3) \otimes R(\alpha(l_2)) \otimes \alpha(l_1) \\
	&- R(l_1)_2 \otimes R(R(l_1)_1) \otimes  \alpha(l_2)+R(l_1)_1 \otimes R(R(l_1)_2) \otimes  \alpha(l_2)+R(l_1)_2 \otimes R(l_1)_1\otimes  \alpha(l_2) \\
	&+R(l_2)_2 \otimes R(R(l_2)_1) \otimes  \alpha(l_1)-R(l_2)_1 \otimes R(R(l_2)_2) \otimes  \alpha(l_1)-R(l_2)_2 \otimes R(l_2)_1\otimes  \alpha(l_1) \\
	&+R(\alpha(l_2)) \otimes \alpha R(l_1) \otimes \alpha(l_3)+R(\alpha(l_2)) \otimes \alpha R(l_3) \otimes \alpha(l_1)
\end{align*}
By \eqref{eq0.1} and $(\alpha \circ R = R \circ \alpha)$ we obtain 
\begin{align*}
	(\Delta_L - \tau^{12} \Delta_L)(l)&= R(l_{11}) \otimes R(l_{12}) \otimes \alpha(l_2) - R(l_{12}) \otimes R(l_{11}) \otimes \alpha(l_2) - R(l_{21}) \otimes R(l_{22}) \otimes \alpha(l_1) \\
	&+ R(l_{22}) \otimes R(l_{21}) \otimes \alpha(l_1)- R(\alpha(l_1)) \otimes R(\alpha(l_2)) \otimes \alpha(l_3) +R(\alpha(l_2)) \otimes R(\alpha(l_1)) \otimes \alpha(l_3) \\
	&+ R(\alpha(l_2)) \otimes R(\alpha(l_3)) \otimes \alpha(l_1) - R(\alpha(l_3)) \otimes R(\alpha(l_2)) \otimes \alpha(l_1) \\
	&=0 
\end{align*}
Which mean $(L,\tilde{\Delta},\alpha)$ is a Hom-pre-Lie coalgebra. We need just to prove that $\varphi_L$ is an $\alpha$-coderivation.
\begin{align*}
	\tilde{\Delta} \circ \varphi_L&=\Big((R \otimes 1_L)\Delta-\tau \circ (1_L \otimes R)\Delta-\Delta \Big) \circ \varphi_L\\
	&=(R \otimes 1_L)(\Delta \circ \varphi_L)-\tau \circ (1_L \otimes R)(\Delta \circ \varphi_L)-(\Delta \circ \varphi_L) \\
	&=(R\otimes 1_L)(\varphi_L \otimes \alpha)\Delta-\tau \circ (1_L \otimes R)(\varphi_L \otimes \alpha)\Delta-(\varphi_L \otimes \alpha)\Delta \\
	&+(R \otimes 1_L)(\alpha \otimes \varphi_L)\Delta-\tau \circ (1_L \otimes R)(\alpha \otimes \varphi_L)\Delta-(\alpha \otimes \varphi_L)\Delta \\
	&\overset{\text{By using proposition \eqref{prop3.3}}}{=}(\varphi_L \otimes \alpha)(R \otimes 1)\Delta-(\varphi_L \otimes \alpha)(\tau \circ (1_L \otimes R)\Delta)-(\varphi_L \otimes \alpha)\Delta \\
	&+(\alpha \otimes \varphi_L)(R \otimes 1)\Delta-(\alpha \otimes \varphi_L)(\tau \circ (1_L \otimes R)\Delta)-(\alpha \otimes \varphi_L)\Delta \\
	&=(\varphi_L \otimes \alpha)	\tilde{\Delta} + (\alpha \otimes \varphi_L)	\tilde{\Delta}.
\end{align*}
This completes the proof. 
\end{proof}
\begin{re}
	Let $(L,\Delta,\alpha)$ be an Hom-associative coalgebra, than for $l \in L$, 
 \begin{align*}
     (\alpha \otimes \Delta) \Delta(l)&=(\Delta \otimes \alpha)\Delta(l) \\
     (\alpha \otimes \Delta)(l_1 \otimes l_2)&=(\Delta \otimes \alpha)(l_1 \otimes l_2) \\
     \alpha(l_1) \otimes l_{21} \otimes l_{22}&=l_{11} \otimes l_{12} \otimes \alpha(l_2)
 \end{align*}
 which mean $\alpha(l_1)=l_{11}, \ l_{21}=l_{12} \text{ and } \ l_{22}=\alpha(l_2)$. \\
 That is why in the previous proof we used the following notations
 \begin{equation*}
     l_{11}=\alpha(l_1) \text{ , } l_{12}=l_{21}=\alpha(l_2) \text{ and } \alpha(l_2)=l_{22}=\alpha(l_3).
 \end{equation*}
\end{re}
In the next theorem, we got the same result of theorem \eqref{thm3.2} but in the case of $\lambda=0$.
\begin{thm} \label{thm3.4}
	Let $(L,\Delta,\varphi_L,\alpha,R)$ be a Rota-Baxter Hom-CoAss-CoDer pair of weight $(\lambda=0)$. We define $\tilde{\Delta}$ on $L$ by 
	\begin{equation*}
		\tilde{\Delta}=(R \otimes 1_L) \Delta - \tau \circ (1_L \otimes R)\Delta,	
	\end{equation*}
	implies that for $l \in L$,
	\begin{equation*}
		\tilde{\Delta}(l)=R(l_1) \otimes l_2 - R(l_2) \otimes l_1.	
	\end{equation*}
	Then $(L,\tilde{\Delta},\varphi_L,\alpha)$ is a Hom-pre-Lie CoDer pair.
	With
	\begin{equation*} 
		R \circ \alpha=\alpha \circ R,
	\end{equation*}
	\begin{equation*}
		\varphi_L \circ R =R \circ \varphi_L.
	\end{equation*}	
\end{thm}
\begin{proof}
   The proof is similar to theorem \ref{thm3.2}.  
\end{proof}

With the previous two results we obtain the following diagram
\begin{center}
	$\xymatrix{
		\textbf{\emph{Hom-CoAss-CoDer pair}}&\underrightarrow{ \text{R.B.O of weight }\lambda=0,-1} &\textbf{\emph{Hom-pre-Lie CoDer pair}} }$
\end{center}
Authors in paper \cite{D1} studied coassociative coalgebras equipped with a coderivation $\delta_L$, which leads to defining \textbf{AssCoDer pair} $(L,\delta_L)$ in \cite{A1}. In this  paper, we study the \textbf{Hom-Ass CoDer pair} $(L,\delta_L,\alpha)$ and its relation with the \textbf{Hom-Lie CoDer pair}. \\
Hom-coassociative coalgebra is a quadruple $(L,\Delta,\delta_L,\alpha)$ where $L$ is a vector space, $\Delta : L \rightarrow L$ a linear map, $\delta_L: L \rightarrow L$ an $\alpha$-coderivation and $\alpha :L \rightarrow L$ the twisting map with 
\begin{center}
	$(\alpha \otimes \Delta) \Delta = (\Delta \otimes \alpha) \Delta$.
\end{center}
Let us start with the following definition 
\begin{defn}
	An \textbf{Hom-Ass CoDer pair} is a coassociative coalgebra $L$ with an $\alpha$-coderivation $\delta_L$, denoted by $(L,\Delta,\delta_L,\alpha)$ .
\end{defn}
\begin{defn}
	Let $(L,\Delta,\delta_L,\alpha)$ and $(K,\Delta,\delta_K,\alpha)$ be two \textbf{Hom-Ass CoDer pairs}. A \textbf{Hom-Ass CoDer} morphism $\psi$ from $L$ to $K$ is a coalgebra morphism $\psi:L \rightarrow k$ such that 
	\begin{center}
		$\delta_L \circ \psi =\psi \circ \delta_K \text{ and } \psi \circ \alpha=\beta \circ \psi$ .
	\end{center}
\end{defn}
\begin{defn}\cite{M4}
	Let $(L,\Delta,\delta_L)$ be an \textbf{Ass-CoDer pair}, the linear map $T: L \rightarrow L$ is said to be endomorphism operator if 
	\begin{equation}
		\Delta \circ T = (T \otimes T) \Delta.
	\end{equation}
\end{defn}
The previous definition can be extended to the case of Hom structure by adding the following condition 
\begin{equation} \label{eq3.5}
    T \circ \alpha = \alpha \circ T.
\end{equation}
Where $(L,\Delta,\delta_L,\alpha)$ be an \textbf{Hom-Ass CoDer pair}. \\
It is well known that any Hom-coassociative coalgebra becomes a Hom-Lie coalgebra via the commutator $\Delta_C=\Delta - \tau \circ \Delta$, the next proposition generalize that results in the case of \textbf{Hom-CoDer pairs}.
\begin{prop}
	Let $(L,\Delta,\delta_L,\alpha)$ be an \textbf{Hom-AssCoDer pair}. Then we can define an \textbf{Hom-Lie CoDer pair}$ (L,\Delta_c,\delta_L,\alpha)$ where :
	\begin{align*}
	\Delta_C=\Delta - \tau \circ \Delta	.
	\end{align*}
\end{prop}
\begin{proof}
	We just need to prove that $\delta_L$ is an 
 $\alpha$-coderivation on  $\Delta_C$. Consider that
 \begin{align*}
 	\Delta_c \circ \delta_L&=(\Delta - \tau \circ \Delta) \circ \delta_L \\
 	&=\Delta \circ \delta_L - \tau \circ \Delta \delta_L \\
 	&=(\alpha \otimes \delta_L) \circ \Delta + (\delta_L \otimes \alpha) \circ \Delta - \tau \circ (\alpha \otimes \delta_L) \circ \Delta- \tau \circ ( \delta_L \otimes \alpha) \Delta \\
 	&=(\alpha \otimes \delta_L) \circ \Delta + (\delta_L \otimes \alpha) \circ \Delta-(\delta_L \otimes \alpha) \tau \circ  \Delta-(\alpha \otimes \delta_L) \tau \circ \Delta \\
 	&=(\alpha \otimes \delta_L)(\Delta-\tau \circ \Delta)+(\delta_L \otimes \alpha)(\Delta-\tau \circ \Delta) \\
 	&=(\alpha \otimes \delta_L) \Delta_c +(\delta_L \otimes \alpha) \Delta_c.
 \end{align*}
This completes the proof.
\end{proof}
The commutator $\Delta_C=\Delta - \tau \circ \Delta$ is not the only way to construct a Hom-Lie CoDer pair, that is why in the next proposition we give another method to do it.
\begin{prop}
	Let $(L,\Delta,\delta_L,\alpha)$ be an \textbf{Hom-AssCoDer pair} and let $T :L \rightarrow L$ be an endomorphism operator satisfying, for $l \in L$
 \begin{eqnarray} \label{eq3.6}
    T^2(l)&=&T(l) , \\
    T \circ \delta_L&=&\delta_L \circ T. \label{eq3.7}
 \end{eqnarray}
Then a \textbf{Hom-Lie CoDer pair} structure is given by 
	\begin{equation*}
		\Delta_c=(1 \otimes T) \Delta-(T \otimes 1) \tau \circ \Delta
	\end{equation*}
\end{prop}
\begin{proof}
	The proof is divided into two parts, we start with the first one in which we prove that $(L,\Delta,\alpha)$ is a Hom coalgebra.
	\begin{align*}
		\Delta_c+\tau \circ \Delta_c&=(1 \otimes T) \Delta-(T \otimes 1) \tau \circ \Delta+\tau \circ ((1 \otimes T) \Delta-(T \otimes 1) \tau \circ \Delta) \\
		&=(1 \otimes T) \Delta-(T \otimes 1) \tau \circ \Delta+\tau \circ (1 \otimes T) \Delta-\tau \circ (T \otimes 1) \tau \circ \Delta \\
		&=(1 \otimes T) \Delta-(T \otimes 1) \tau \circ \Delta+ (T \otimes 1)\tau \circ \Delta-(1 \otimes T) \Delta \\
		&=0.
	\end{align*}
Here we used the fact that $\tau \circ (1\otimes T) \Delta=(T \otimes 1) \tau \circ \Delta$. \\
Indeed for $l,l_1,l_2 \in L$,
\begin{align*}
	\tau \circ (1\otimes T) \Delta(l)=\tau \circ (l_1 \otimes T(l_2)) 
	=T(l_2) \otimes l_1 
	=(T \otimes 1)(l_2 \otimes l_1) 
	=(T \otimes 1) \tau \circ \Delta(l).
\end{align*}
Which means that skew symmetry holds. Now for the co-Jacobi identity we have 
\begin{align*}
(\alpha \otimes \Delta_c) \circ \Delta_c(l)&=(\alpha \otimes \Delta_c)((1 \otimes T)\Delta(l)-(T \otimes 1) \tau \circ \Delta(l)) \\
&=(\alpha \otimes \Delta_c)(l_1 \otimes T(l_2)-T(l_2) \otimes l_1) \\
&=\alpha(l_1) \otimes ((1 \otimes T)\Delta(T(l_2))-(T \otimes 1) \tau \circ \Delta(T(l_2))) \\
&-\alpha T(l_2) \otimes ((1 \otimes T)\Delta(l_1)-(T \otimes 1)\tau \circ \Delta(l_1)). 
\end{align*}
By the identity $(T \otimes T)\circ \Delta=\Delta \circ T$, $\alpha \circ T=T \circ \alpha$ and $T^2(l)=T(l)$, we obtain that 
\begin{align*}
	(\alpha \otimes \Delta_c) \circ \Delta_c(l)&=\alpha (l_1) \otimes T(l_{21}) \otimes T(l_{22})-\alpha (l_1) \otimes T(l_{22}) \otimes T(l_{21}) \\
	&-T(\alpha (l_2)) \otimes l_{11} \otimes T(l_{12})+T(\alpha (l_2)) \otimes T(l_{12}) \otimes l_{11}
\end{align*}
and by using the identity of coassociative coalgebra, we obtain 
\begin{align*}
	(\alpha \otimes \Delta_c) \circ \Delta_c(l)&=\alpha (l_1) \otimes T(l_{12}) \otimes T(\alpha (l_2))-\alpha (l_1) \otimes T(\alpha (l_2)) \otimes T(l_{12}) \\
	&-T(\alpha (l_2)) \otimes \alpha (l_1) \otimes T(l_{12})+T(\alpha (l_2)) \otimes T(l_{12}) \otimes \alpha (l_1).
\end{align*}
So 
\begin{align*}
	(1 + \xi+\xi^2)(\alpha \otimes \Delta_c) \circ \Delta_c(l)&=\alpha (l_1) \otimes T(l_{12}) \otimes T(\alpha (l_2))-\alpha (l_1) \otimes T(\alpha (l_2)) \otimes T(l_{12}) \\
	&-T(\alpha (l_2)) \otimes \alpha (l_1) \otimes T(l_{12})+T(\alpha (l_2)) \otimes T(l_{12}) \otimes \alpha (l_1) \\	
	&+  T(l_{12}) \otimes T(\alpha (l_2)) \otimes \alpha (l_1)-T(\alpha (l_2)) \otimes T(l_{12}) \otimes \alpha (l_1) \\
	&-\alpha (l_1) \otimes T(l_{12}) \otimes T(\alpha (l_2))+T(l_{12}) \otimes \alpha (l_1) \otimes T(\alpha (l_2))  \\
	&+ T(\alpha (l_2)) \otimes \alpha(l_1)\otimes T(l_{12}) - T(l_{12}) \otimes \alpha(l_1) \otimes T(\alpha (l_2)) \\
	&- T(l_{12}) \otimes T(\alpha (l_2)) \otimes \alpha (l_1) +\alpha (l_1) \otimes T(\alpha (l_2)) \otimes T(l_{12}) \\
	&=0.
\end{align*}
Similarly, we can prove Hom-co-Jacobi's identity. \\
For the second part of the proof, we need to prove that $\delta_L$ is an $\alpha$-coderivation.
\begin{align*}
	\Delta_c \circ  \delta_L&=((1 \otimes T)\Delta-(T\otimes 1)\tau \circ \Delta) \circ \delta_L \\
	&=(1 \otimes T)\Delta \circ \delta_L-(T\otimes 1)\tau \circ \Delta \circ \delta_L \\
	&=(1 \otimes T)(\alpha \circ \delta_L)\Delta+(1 \otimes T)(\delta_L \otimes \alpha)\Delta-(T\otimes 1)(\delta_L \otimes \alpha)\tau \circ \Delta-(T\otimes 1)(\alpha \otimes \delta_L)\tau \circ \Delta . 
\end{align*}
By Eqs \eqref{eq3.5} and \eqref{eq3.7} and the Proposition  \eqref{prop3.3}, we get
\begin{align*}
	\Delta_c \circ  \delta_L&=(\alpha \circ \delta_L)(1 \otimes T)\Delta+(\delta_L \otimes \alpha)(1 \otimes T)\Delta-(\delta_L \otimes \alpha)(T\otimes 1)\tau \circ \Delta-(\alpha \otimes \delta_L)(T\otimes 1)\tau \circ \Delta \\
	&=(\alpha \circ \delta_L)((1 \otimes T)\Delta-(T\otimes 1)\tau \circ \Delta)+(\delta_L \otimes \alpha)((1 \otimes T)\Delta-(T\otimes 1)\tau \circ \Delta) \\
 &=(\alpha \circ \delta_L) \Delta_c + \delta_L \otimes \alpha \Delta_c.
\end{align*}
This completes the proof.
\end{proof}\noindent {\bf Acknowledgment:}
	The authors would like to thank the referee for valuable comments and suggestions on this article.
	

\begin{thebibliography}{999}
			\bibitem{A1} Asif, S., Basdouri, I., Sadraoui, A., Makhlouf, A., 2023. On Lie and pre-Lie CoDer pairs. In preparation.
		
		\bibitem{A2} Asif, S., Wang, Y. and Wu, Z., 2023. RB-operator and Nijenhuis operator of Hom-associative conformal algebra. Journal of Algebra and its Applications.
		
		\bibitem{B1}Bakayoko, I., 2014. L-modules, L-comodules, and Hom-Lie quasi-bialgebras.
		
		\bibitem{C1} Cai, L. and Sheng, Y., 2018. Purely Hom-Lie bialgebras. Science China Mathematics, 61, pp.1553-1566.
		
		\bibitem{D1}Du, L. and Tan, Y., 2021. Coderivations, abelian extensions and cohomology of Lie coalgebras. Communications in Algebra, 49(10), pp.4519-4542.
		
		\bibitem{D2}Du, L., Ma, Y., Xv, J. and Bao, Y., 2023. Cohomologies and deformations of coassociative coderivations. Communications in Algebra, 51(7), pp.3020-3041.
		
		\bibitem{H1}Hartwig, J.T., Larsson, D. and Silvestrov, S.D., 2006. Deformations of Lie algebras using $\sigma$-derivations. Journal of Algebra, 295(2), pp.314-361.
		
		\bibitem{M1} Majid, S., 1997. Foundations of Quantum Group Theory,(Cambridge University, 1995). AU Klimyk and K. Schmüdgen, Quantum Groups, and their Representations.
		
		\bibitem{M5}Makhlouf, A. and Silvestrov, S., 2010. Hom-algebras and Hom-coalgebras. Journal of Algebra and its Applications, 9(04), pp.553-589.
		
		\bibitem{M6}Makhlouf, A. and Silvestrov, S.D., 2008. Hom-algebra Structures. Journal of Generalized Lie Theory and Applications, 2(2), pp.51-64.
		
		\bibitem{M4}Martinez, W.A. and Ceron, S.I., 2023. Construction of some non-associative algebras from associative algebras with a endomorphism operator, differential operator or left averaging operator. arXiv preprint arXiv:2301.11770.
		
		\bibitem{M3}Ma, T. and Liu, L., 2016. Rota–Baxter coalgebras and Rota–Baxter bialgebras. Linear and Multilinear Algebra, 64(5), pp.968-979. 
		
		\bibitem{M2}Michaelis, W., 1980. Lie coalgebras. Advances in Mathematics, 38(1), pp.1-54.
		
		\bibitem{S1}Larson, R.G. and Sweedler, M.E., 1969. An associative orthogonal bilinear form for Hopf algebras. American Journal of Mathematics, 91(1), pp.75-94.
		
		\bibitem{S2} Sheng, Y., 2012. Representations of hom-Lie algebras. Algebras and Representation Theory, 15(6), pp.1081-1098.
		
		\bibitem{T2}Tang, R., Frégier, Y. and Sheng, Y., 2019. Cohomologies of a Lie Algebra with a derivation and applications. Journal of Algebra, 534, pp.65-99.
		
		\bibitem{T1}Turaev, V., 2005. Loops on surfaces, Feynman diagrams, and trees. Journal of Geometry and Physics, 53(4), pp.461-482.

       \bibitem{Y1} Yau, D., 2015. The classical hom-Yang-Baxter equation and hom-Lie bialgebras. International Electronic Journal of Algebra, 17(17), pp. 11-45.
		
		\bibitem{Y2}Yau, D., 2009. Hom-Algebras and Homology. Journal of Lie Theory, 19, pp.409-421.
		
		
\end{thebibliography}
\end{document}